\documentclass{amsart}
\usepackage{amsmath, amscd, amsfonts, latexsym, amsthm, amsxtra,  amssymb, amscd, rotating, epsfig}
\usepackage{bbm}
\usepackage{enumerate}
\usepackage[all]{xy}

\usepackage{hyperref}
\hypersetup{colorlinks=false}
\usepackage{graphicx}

\newtheorem{theorem}{Theorem}[section]
\newtheorem{proposition}[theorem]{Proposition}
\newtheorem{lemma}[theorem]{Lemma}

\numberwithin{equation}{section}

\theoremstyle{remark}
\newtheorem{example}[theorem]{Example}
\newtheorem{remark}[theorem]{Remark}

\theoremstyle{definition}
\newtheorem{definition}[theorem]{Definition}

\DeclareMathOperator*{\Lie}{Lie}
\DeclareMathOperator*{\im}{Image}
\DeclareMathOperator*{\Ind}{Ind}
\DeclareMathOperator*{\reg}{reg}
\DeclareMathOperator*{\el}{ell}

\DeclareMathOperator*{\Irr}{Irr}
\DeclareMathOperator*{\Stab}{Stab}
\DeclareMathOperator*{\ic}{\textsf{inf.ch.}}
\DeclareMathOperator*{\disc}{disc}
\DeclareMathOperator*{\temp}{temp}
\DeclareMathOperator*{\ad}{ad}

\newcommand{\dsum}{\displaystyle\sum}

\DeclareMathOperator*{\Hom}{Hom} 

\DeclareMathOperator*{\Vol}{Vol}
\DeclareMathOperator*{\St}{\mathsf{St}}

\newcommand{\s}{\simeq}
\newcommand{\pp}{\mathfrak{p}}

\newcommand{\OO}{\mathcal{O}}

\newcommand{\ma}{\mathfrak{a}}

\newcommand{\CC}{\mathbb{C}}

\newcommand{\RR}{\mathbb{R}}
\newcommand{\QQ}{\mathbb{Q}}
\newcommand{\ZZ}{\mathbb{Z}}

\newcommand{\rr}{\mathbf{\textit{r}}}
\newcommand{\ii}{\mathbf{\textit{i}}}

\title[Generalized pseudo-coefficients]{Generalized pseudo-coefficients of discrete series of $p$-adic groups}

\author[Kwangho Choiy]{Kwangho Choiy} 
\address{Department of Mathematics,
Oklahoma State University,
Stillwater, OK 74078-1058,
U.S.A.}
\email{kwangho.choiy@okstate.edu}


 \date{\today$~$-- \textsf{Published in Asian J. Math. 20 (2016), no. 5, pp.969-988.}}

\begin{document}
\maketitle  

\begin{abstract}
Let $G$ be a connected reductive group over a $p$-adic field $F$ of characteristic 0 and let $M$ be an $F$-Levi subgroup of $G.$
Given a discrete series representation $\sigma$ of $M(F),$
we prove that there exists a locally constant and compactly supported function on $M(F),$ which generalizes a pseudo-coefficient of $\sigma.$
This function satisfies similar properties to the pseudo-coefficient, and its lifting to $G(F)$ is applied to the Plancherel formula.
\end{abstract}
\setcounter{tocdepth}{2}
\tableofcontents

\section{Introduction} \label{intro}
Given a connected \textit{semi-simple} group $G$ over a $p$-adic field $F$ of characteristic 0 and a discrete series representation $\sigma$ of $G(F)$ (that is, $\sigma$ is an irreducible, admissible, unitary representation whose matrix coefficients are square-integrable modulo the center of $G(F)$), 
it is well-known (\cite{kaz}, \cite{ss}) that there exists a pseudo-coefficient $\phi_\sigma$ for $\sigma,$ which lies in the Hecke algebra $C^{\infty}_{c}(G(F))$ of locally constant and compactly supported functions on $G(F),$ satisfying the property that,
for any tempered representation $\pi$ of $G(F),$
\[
  trace(\pi)(\phi_{\sigma}) = \left\{ 
  \begin{array}{l l}
    1, & \: \text{if} \, \, \pi  \s  \sigma, \\

    0, & \: \text{otherwise.}\\
  \end{array} \right.
\]
Historically, Harish-Chandra already knew that the Howe conjecture implied the existence (see \cite[p.102]{clo89}).
Another proof of the existence relies on the trace Paley-Wiener theorem (\cite{bdk}).
As an explicit example of such a pseudo-coefficient, $\phi_{\St}$ with the Steinberg representation  $\St$ of $G(F)$ forms an Euler-Poincar\'e function in \cite{kot88}, and this result is generalized to any discrete series representation of $G(F)$ in \cite{ss}.

The pseudo-coefficient $\phi_{\sigma}$ satisfies interesting properties.
The orbital integral of $\phi_{\sigma}$ at a regular semi-simple element $\gamma$ of $G(F)$ vanishes unless $\gamma$ is elliptic (\cite{kaz}, \cite{ss}). 
This property, so-called cuspidality, is so crucial as to get various trace formulas (see \cite{harris11} for example). 
It is also applied to endoscopic transfers (\cite{ac89}, \cite{hs11}) and the limit multiplicities of discrete series representations in spaces of automorphic forms (\cite{clo86}, \cite{flm}). The orbital integral is also related to the Harish-Chandra character function of $\sigma$ (\cite{kaz}, \cite{ss}).

All work on pseudo-coefficients under the condition `semi-simplicity' can be extended to the case that $G(F)$ has compact center. This is possible because the group $\Psi(G)$ of unramified characters on $G(F)$ is finite (see Remark \ref{properties of M/M1}). 
One then chooses a regular function $\xi$ on the $\Psi(G)$-orbit of $\sigma$ such that $\xi(z)$ vanishes unless $z=\sigma.$
This $\xi$ thus gives the pseudo-coefficient $\phi_{\sigma}$ by means of the trace Paley-Wiener theorem (\cite{bdk}).

In the general case that the group $G(F)$ has non-compact center, however, $\Psi(G)$ is no longer finite. 
It follows that there exists no such $\xi$ as above.
For such a group, instead, one can find a function $\phi^{\omega}_{\sigma}$ in the Hecke algebra $C^{\infty}_{c}(G(F), \omega^{-1})$ of locally constant functions with compact support \textit{modulo the center} and transforming under the center by $\omega^{-1},$ where $\omega$ is the central character of $\sigma.$
The function $\phi^{\omega}_{\sigma}$ satisfies the property that, for any tempered representation $\pi$ of $G(F)$ whose central character is $\omega,$
\[
  trace(\pi)(\phi^{\omega}_{\sigma}) = \left\{ 
  \begin{array}{l l}
    1, & \: \text{if} \, \, \pi  \s  \sigma, \\

    0, & \: \text{otherwise.}\\
  \end{array} \right.
\]
We also say such a function $\phi^{\omega}_{\sigma}$ is a pseudo-coefficient of $\sigma$ (see \cite{clo89}, \cite[A.4]{dkv}).

The purpose of this paper is to study a locally constant and compactly supported function on $G(F)$ with non-compact center, which generalizes both pseudo-coefficients $\phi_\sigma \in C^{\infty}_{c}(G(F))$ for the case with compact center and $\phi^{\omega}_{\sigma} \in C^{\infty}_{c}(G(F), \omega^{-1})$ for the case with non-compact center. 
We work with an arbitrary $F$-Levi subgroup $M$ (possibly, $M=G$) of a connected reductive group $G$ over $F.$ 
Our generalized pseudo-coefficient is determined by a discrete series representation of $M(F)$ and a regular function on the $\Psi(G)$-orbit of $\sigma.$
It also satisfies the cuspidal property.
Furthermore, its projection to $C^{\infty}_{c}(M(F), \omega^{-1})$ is related to the pseudo-coefficient $\phi^{\omega}_{\sigma}$ and the Harish-Chandra character function of $\sigma.$
As an application, we lift it to a function in $C^{\infty}_{c}(G(F))$ and simplify the Plancherel formula.

We explain our results more precisely.
Let $\sigma$ be a discrete series representation  of $M(F)$ and let $\xi$ be a regular function on the $\Psi(M)$-orbit of $\sigma$ (denoted by $\Omega(\sigma)$). 
Using the trace Paley-Wiener theorem (\cite{bdk}), we prove that there exists a locally constant and compactly supported function $\phi_{\sigma, \xi}$ on $M(F)$ such that, for all tempered representations $\pi$ of $(M(F))$,
\[
  trace(\pi)(\phi_{\sigma, \xi}) = \left\{ 
  \begin{array}{l l}
    \xi(\pi), & \: \text{if} ~ \pi \s \sigma \otimes \psi ~ \text{for some} ~ \psi \in  \Psi(M), \\

    0, & \: \text{otherwise.}\\
  \end{array} \right.
\]
We call $\phi_{\sigma, \xi}$ \textit{a generalized pseudo-coefficient for $\sigma$ and $\xi$}. 
We remark that the regular function $\xi(z)$ cannot be chosen to be 0 for $z \neq \sigma,$ unless $M=G$ and $G(F)$ has compact center.
The reason is that $\Omega(\sigma)$ is isomorphic to the quotient of $(\CC^{\times})^d$ by a finite subgroup, the stabilizer ${\Stab}_{\Psi}(\sigma)$ of $\sigma$ in $\Psi(M),$ 
where $d$ is the dimension of the split component in the center $Z_M$ of $M$ (see Sections \ref{unram} and \ref{reg} for details). 
This part is different from the case with compact center. 
Furthermore, the fact that $\phi_{\sigma, \xi}$ has compact support on $M(F)$ makes a distinction from $\phi^{\omega}_{\sigma} \in C^{\infty}_{c}(M(F), \omega^{-1}).$

Despite such differences the generalized pseudo-coefficient $\phi_{\sigma, \xi}$ satisfies several properties similar to those pseudo-coefficients $\phi_{\sigma} $ and $\phi^{\omega}_{\sigma}.$
Given $\sigma$ and $\xi,$ the function $\phi_{\sigma, \xi}$ is uniquely determined modulo the $\CC$-subspace that is spanned by the functions of the form $m \mapsto (h(m)-h(xmx^{-1})),$ for all $h \in C^{\infty}_{c}(M(F))$ and $x \in M(F)$ (see Remark \ref{uniquely determined}).
The function also satisfies the cuspidal property (Proposition \ref{cuspidal property})
\[
\mathit{O}^{M(F)}_{\gamma}(\phi_{\sigma, \xi}) = 0 = \mathit{O}^{M(F)}_{\gamma}(\phi_{\sigma, \xi}^{\omega}),
\]
for any $\gamma$  regular semi-simple but not elliptic in $M(F),$ where $\phi_{\sigma, \xi}^{\omega}(m)$ denotes the canonical projection 
$\int _{Z_M(F)} \omega(z) \phi_{\sigma, \xi}(zm) dz$ of $\phi_{\sigma, \xi}$
onto $C^{\infty}_{c}(M(F), \omega^{-1}).$
Furthermore, by making a suitable choice of $\xi$ so that $\xi(z)$ vanishes on a finite subset  of $\Omega(\sigma)$ whose central character is $\omega,$ unless $z = \sigma$ (see Lemma \ref{some lemma for finiteness} for the finite subset),
Proposition \ref{pro for relation bw character and orb int} provides a relationship between the orbital integral of  $\phi_{\sigma, \xi}^{\omega}$ and the Harish-Chandra character $\Theta_{\sigma}$
\[
\mathit{O}^{M(F)}_{\gamma}(\phi_{\sigma, \xi}^{\omega}) = \xi(\sigma) \overline{\Theta_{\sigma}(\gamma)},
\]
for every elliptic regular semi-simple $\gamma$ in $M(F).$ 
Here the bar $\, \bar{~~~} \,$ means the complex conjugation.
We note that the value $\xi(\sigma)$ above is determined by $h^{\omega}_{\sigma}(1)$ and the formal degree $d_M(\sigma),$ 
where $h^{\omega}_{\sigma}$ is a matrix coefficient of $\sigma$ (see Remark \ref{rem for the value}).
By taking the regular function $\xi$ with $\xi(\sigma)=1,$ that is, $h^{\omega}_{\sigma}(1)= d_M(\sigma),$ we get a simple relationship between $\phi_{\sigma, \xi}^{\omega}$ and a pseudo-coefficient $\phi_{\sigma}^{\omega}$ of $\sigma$ (see Proposition \ref{clozel}).
Applying $\phi_{\sigma, \xi}^{\omega}$ to the Plancherel formula, we also get
\[
\phi_{\sigma, \xi}(1) = d_M(\sigma)  \int_{\psi \in \Psi_u(M) / {\Stab}_{\Psi}(\sigma)}  \xi(\sigma \otimes \psi) d \psi,
\]
where $\Psi_u(M)$ is the subgroup of $\Psi(M)$ consisting of unitary characters (Proposition \ref{formula of phi(1)}).

Using the cuspidal property and parabolic descent, we lift the generalized pseudo-coefficient $\phi_{\sigma, \xi} \in C^{\infty}_{c}(M(F))$ to $f_{\sigma, \xi} \in C^{\infty}_{c}(G(F)).$ 
They are related by the formula (Theorem \ref{pro for phi and f})
\[
|D_{G/M}(\gamma)| ^{1/2} \mathit{O}^{G(F)}_{\gamma}(f_{\sigma, \xi}) 
= \mathit{O}^{M(F)}_{\gamma}(\overline{f}_{\sigma, \xi}^{(P)})
= \mathit{O}^{M(F)}_{\gamma}(|W(\theta, \theta)| \cdot \phi_{\sigma, \xi}),
\]
for every elliptic regular semi-simple $\gamma$ in $M(F),$ where $D_{G/M}(\gamma) = \det(1-ad(\gamma)) |_{\Lie(G)/\Lie(M)}.$
Here we impose the condition that $\xi$ be a $W(\theta, \theta)$-invariant regular function on $\Omega(\sigma).$
We refer to Section \ref{lifting} for unexplained notations.
As an application, this lifting $f_{\sigma, \xi}$ turns out to simplify the Plancherel formula as in Proposition \ref{formula of f(1)}.

This paper is organized as follows. In Section \ref{prelim}, we recall background material and review basic facts on the Plancherel formula, unramified characters, and regular functions on $\Psi(M)$ and $\Omega(\sigma).$
In Section \ref{sec2}, we prove the existence of a generalized pseudo-coefficient and study its properties: cuspidality, a relation to the Harish-Chandra character function as well as `usual' pseudo-coefficients. 
We also lift the generalized pseudo-coefficient in $C^{\infty}_{c}(M(F))$ to a function in $C^{\infty}_{c}(G(F))$ and apply the generalized pseudo-coefficient and the lift to the Plancherel formula. 

\section{Preliminaries} \label{prelim}
This section is devoted to basic notions, terminologies and known results. We mainly refer to \cite{hai10}, \cite{hc73}, \cite{kot86}, \cite{sh90}, \cite{shin10}, and \cite{wal03}.

\subsection{Notation and convention} \label{notation and convention}
Let $F$ be a $p$-adic field of characteristic $0,$ that is, a finite extension of $\QQ_{p},$ where $p$ is a prime number. 
Let $\bar{F}$ be an algebraic closure of $F.$ Denote by $\OO_F$ the ring of integers of $F$ and by $\pp$ the maximal ideal in $\OO_F.$ Let $q$ denote the cardinality of the residue field $\OO_F / \pp$ and $|\cdot|_F$ the normalized absolute value on $F.$ 

Let $G$ be a connected reductive algebraic group over $F.$ 
Write $G(F)$ for the group of $F$-points of $G.$
For a semi-simple element $\gamma \in G(F),$ write $G_{\gamma}$ for the centralizer in $G$ of $\gamma$ and $G_{\gamma}^{\circ}$ for its identity component. 
We say $\gamma$ is regular if $G_{\gamma}^{\circ}$ is a maximal torus in $G.$ 
Let $G(F)^{\reg}$ denote the set of regular semi-simple elements in $G(F).$  
We say $\gamma$ is elliptic if the quotient $G_{\gamma}^{\circ} / Z_G^{\circ}$ is anisotropic over $F.$
Let $G(F)^{\el}$ denote the set of elliptic regular semi-simple elements in $G(F).$

We denote by $C^{\infty}_{c}(G(F))$ the Hecke algebra of locally constant functions with compact support. 
Let $\Irr(G(F))$ (resp., $\Irr_u(G(F))$) be the set of equivalence classes of irreducible admissible (resp., unitary) representations of $G(F).$ 
By abuse of notation, we identify an equivalence class with its representative.
Let $\Pi_{\temp}(G(F))$ and $\Pi_{\disc}(G(F))$ be the subsets of $\Irr(G(F))$ consisting of irreducible tempered representations and discrete series representations of $G(F),$ respectively.
Here, we say an admissible representation is a discrete series representation if it is in $\Irr_u(G(F))$ and has non-zero matrix coefficients which are square-integrable modulo the center of $G(F).$ 
Since it is irreducible, if one non-zero matrix coefficient is square-integrable, then so are all the others.

Fix a minimal $F$-parabolic subgroup $P_0$ of $G$ with its Levi decomposition $M_0 N_0,$ where $M_0$ is a minimal $F$-Levi subgroup and $N_0$ is the unipotent radical of $P_0.$ 
Let $A_0$ be the split component of $M_0$ (or $P_0$), that is, the maximal $F-$split torus in the center of $M_0.$ Denote by $W_G = W(G, A_0) := N_G(A_0) / Z_G(A_0)$ the Weyl group of $A_0$ in $G,$ where $N_G(A_0)$ and $Z_G(A_0)$ are the normalizer and the centralizer of $A_0$ in $G,$ respectively.
We denote by $\Phi$ the set of roots of $G$ with respect to $A_0.$ Let $\Phi^+ $ be the set of positive roots. Note that the choice of $P_0$ determines the set $\Phi^+$ of positive roots. Let $\Delta \subset \Phi^+$ denote the set of simple roots.

Let $P$ be a standard $F$-parabolic subgroup (that is, $P \supseteq P_0$) with its Levi factor $M=M_{\theta} \supseteq M_0$ generated by a subset $\theta \subseteq \Delta$ and its unipotent radical $N \subseteq N_0.$ 
Let $A_{M}$ be the split component of $M.$ Then we write $\Phi_{\theta}$ for the subset of $\Phi$ consisting of $\ZZ-$linear combinations of the roots in $\theta,$ and $\Phi^+_{\theta}$ for $\Phi_{\theta} \cap \Phi^+.$ 
Denote by $W_M = W(G, A_M) := N_G(A_M) / Z_G(A_M)$ the Weyl group of $A_M$ in $G.$ We denote by $P^- = M N^-$ the parabolic subgroup opposite to $P,$ where $N^-$ is opposite to $N.$ For a standard $F$-Levi subgroup $L = L_{\vartheta},$ we set $W(\vartheta, \theta):= \lbrace w \in W_G : w(\vartheta) = \theta \rbrace.$

Let $X^{*}(M)_F$ be the group of $F$-rational characters of $M.$ We denote by
$
\ma_M := \Hom(X^{*}(M)_F ,\RR) = \Hom(X^{*}(A_M)_F , \RR)
$
the real Lie algebra of $A_{M}.$ Let
$
\mathfrak{a}^{*}_{M, \CC} := \mathfrak{a}^{*}_{M} \otimes_{\RR} \CC
$
be the complex dual of $\ma_M,$ where
$
\ma^{*}_{M} := X^{*}(A_{M})_F \otimes_{\ZZ} \RR = X^{*}(M)_F \otimes_{\ZZ} \RR
$
denotes the dual of $\ma_M.$ We define the homomorphism $H_{M} : M(F) \rightarrow \mathfrak{a}_{M}$ by 
\[
q ^{\langle  \chi, H_{M}(m) \rangle} = |\chi(m)|_F
\]
for all $\chi \in X^{*}(M)_F$ and $m \in M(F).$ Fix a maximal special compact (parahoric) subgroup $K$ of $G(F).$ Then we have the Iwasawa decomposition $G(F)=P(F)K.$ Note that one can extend $H_M$ to $G(F)$ using the Iwasawa decomposition by $q ^{\langle  \chi, H_{M}(mnk) \rangle} = |\chi(m)|_F.$

Let $\delta_P: M(F) \rightarrow \RR^{\times}_{>0}$ be the modulus character of $P$ defined by 
$
m \mapsto |\det(\ad(m))|_{\Lie(P) / \Lie(M)}|_F.
$ 
Note that $\delta_P(m)=q ^{\langle  2\rho_P, H_{M}(m) \rangle},$ where $\rho_P:= \frac{1}{2} \sum_{\alpha \in \Phi^+ \smallsetminus \Phi^{+}_{\theta}} \alpha.$ Moreover, $\delta_P$ can be extended to $G(F)$ in the same way as $H_M.$

Now we define the normalized induced representation and the intertwining operator. 
For $\sigma \in \Irr(M(F))$ and $\nu \in \ma^{*}_{M, \CC},$ we denote by $I(\nu, \sigma)$ the normalized induced representation 
\begin{equation*} \label{ind}
I(\nu, \sigma) = {\Ind}_{P(F)} ^{G(F)} (\sigma \otimes q ^{\langle  \nu, H_{M}(\,) \rangle} \otimes \mathbbm{1}).
\end{equation*}
The space $V(\nu, \sigma)$ of $I(\nu, \sigma)$ consists of locally constant functions $f$ from $G(F)$ into the representation space of $\sigma$ such that 
$
f(mng) = \sigma(m) q ^{\langle  \nu + \rho_P, H_{M}(m) \rangle} f(g),
$
for $m \in M(F),$ $n \in N(F)$ and $g \in G(F).$ The group $G(F)$ acts on $V(\nu, \sigma)$ by the regular right action. 
We often write $\ii_{G,M}(\nu, \sigma)$ for $I(\nu, \sigma)$ in order to specify groups.

\subsection{Plancherel formula} \label{def of PM}
We review the Plancherel formula. 
We continue with the notation of Section \ref{notation and convention}.
Fix a representative $w \in G$ of $\tilde{w} \in W_G$ such that $\tilde{w}(\theta) \subseteq \Delta.$ 
For simplicity we often omit $\sim$ for the representative (see \cite[p.280]{sh90}). 
Denote $N_{\tilde{w}} :=N_0 \cap wN^{-}w^{-1},$ that is, the subgroup of $N_0$ generated by the root groups $U_{\alpha},$ $\alpha \in \Phi ^+$ such that $w^{-1}\alpha \in N^-.$ 
We fix a Haar measure $dn$ on $N_{\tilde{w}}.$ Given $f \in V(\nu, \sigma),$ for $g \in G(F),$ the standard intertwining operator is defined as
\begin{equation*} \label{intertwining}
A(\nu, \sigma, \tilde{w}) f(g) = \int_{N_{\tilde{w}}(F)} f(w^{-1} n g) dn.
\end{equation*}
Note that the integral converges absolutely if $Re \langle \nu, \alpha^{\vee} \rangle \gg 0, \; \forall \alpha \in \Delta \smallsetminus \theta,$ where $\alpha^{\vee}$ is the coroot attached to $\alpha.$ 
Moreover, it turns out that $A(\nu, \sigma, \tilde{w}) : I(\nu, \sigma) \rightarrow I(\tilde{w} (\nu), \tilde{w}(\sigma)),$ where $\tilde{w}(\sigma)(m_w):=\sigma(w^{-1}m_w w)$ for $m_w \in M_{\tilde{w}(\theta)}(F)$ with $M_{\tilde{w}(\theta)}=wMw^{-1}.$ We set
\[
\gamma _{\tilde{w}}(G|M) = \int_{\bar{N}_{\tilde{w}}(F)} q ^{\langle 2 \rho_P, H_{M}(\bar{n}) \rangle} d\bar{n},
\]
where $\bar{N}_{\tilde{w}}:=w^{-1}N_{\tilde{w}}w=N^- \cap w^{-1}N_0 w$ and $d\bar{n}$ is the normalized Haar measure on $\bar{N}_{\tilde{w}}(F)$ so that the measure of $K_0 \cap \bar{N}_{\tilde{w}}(F)$ is one. 
Here, $K_0$ is a maximal special parahoric subgroup of $G(F).$

We define the Plancherel measure, due to \cite[Theorem 20]{hc73}, as follows.
\begin{definition} \label{df PM}
For $\nu \in \ma^{*}_{M, \CC},$ $\sigma \in \Irr_u(M(F))$ and $\tilde{w}(\theta) \subseteq \Delta,$ \textit{the Plancherel measure} attached to  $\nu,$ $\sigma$ and $\tilde{w}$ for $\ii_{G,M}(\nu, \sigma)$ is defined as a non-zero complex number $\mu_{M}(\nu, \sigma, w)$ such that
\[
A(\nu, \sigma, \tilde{w}) A(\tilde{w} (\nu), \tilde{w}(\sigma), \tilde{w}^{-1})= \mu_{M}(\nu, \sigma, w)^{-1} 
\gamma _{\tilde{w}}(G|M)^{2}.
\]
\end{definition}
The Plancherel measure $\mu_{M}(\nu, \sigma, w)$ depends only on $\nu$, $\sigma$ and $\tilde{w}$. It is independent of the choices of any Haar measure and of any representative $w$ of $\tilde{w}$ (\cite[p.280]{sh90}). 
\begin{example}
Suppose $F=\QQ_p,$ $G=SL_2,$ $M \s GL_1$ (the diagonal matrix) and $w = \bigl(\begin{smallmatrix}
0&1\\ -1&0
\end{smallmatrix} \bigr).$ Denote by $\chi$ a unitary unramified character of $M(F).$ For $s \in \CC,$ we identify $s$ with $s \tilde{\alpha} \in \ma^{*}_{M, \CC} \s \CC,$ where $\tilde{\alpha}=\frac{1}{2}(e_1 - e_2).$ Then the Plancherel measure $\mu_{M}(s, \chi, w)$ attached to $s,$ $\chi$ and $w$ is
\[
(1+\frac{1}{p})^2 \cdot
\frac{1-\chi(p)p^{-s}}{1 - \chi(p) p^{-s-1}} \cdot
\frac{1-\chi(p)^{-1}p^{s}}{1 - \chi(p)^{-1} p^{s-1}},
\]
which lies in the field $\CC  ( p^{-s})$ of rational functions in $p^{-s}.$
\end{example}

We now recall the Plancherel formula of Harish-Chandra. 
We refer to \cite[Section 16]{hc84} and \cite[Theorem VIII.1.1]{wal03}. 
We normalize Haar measures as in \cite[Section 5]{hc84}. 
Fix a maximal special parahoric subgroup $K_0$ of $G(F),$ which is in a good relative position with $M_0$ (cf. \cite[Section 0.6]{sil79}). 
For any closed subgroup $H(F)$ of $G(F)$ with a Haar measure $dh$ on $H(F),$ we normalize $dh$ so that the volume $\Vol_{dh}(H(F) \cap K_0) = 1.$ 
Then, Harish-Chandra's $\gamma$-factor $\gamma(G|M)$ and Harish-Chandra's $c$-factor $c(G|M)$ are defined as
\begin{equation*} \label{gamma(G|M)}
\gamma(G|M) 
:= \int_{N^{-}} q ^{\langle 2 \rho_P, H_{M}(n^{-}) \rangle} dn^{-},
\end{equation*}
\[
c(G|M) :=  \gamma(G|M)^{-1} \prod_{\alpha \in \Phi(P, A_M)} \gamma(M_{\alpha}|M).
\]
Set 
\[
a(G|M) := c(G|M)^{-2} \gamma(G|M) ^{-1} |W_M|^{-1}.
\] 

For $\pi \in \Pi_{\disc}(M(F)),$ we denote by $d_M(\pi)$  the formal degree of $\pi$ with respect to the Haar measure $dm$ on $M(F)$ 
and by $\mu_M (\pi)$ the Plancherel measure $\mu_M (0,\pi,w)$
as in Definition \ref{df PM}.
We write $d \pi$ for the measure at $\pi \in \Pi_{\disc}(M(F))$ that is transferred from the normalized Haar measure $d\psi$ on the quotient $\Psi_u(M) / {\Stab}_{\Psi}(\pi)$ (see Sections \ref{unram} and \ref{reg} for definitions) via the natural isomorphism 
\[
\Psi_u(M) / {\Stab}_{\Psi}(\pi) \s \{ \pi \otimes \psi : \psi \in \Psi_u(M) \}
\]
(see \cite[Section 5.4]{mt11} and \cite[Section 2]{hc84}). 
The following proposition states the Plancherel formula.
\begin{theorem} (Harish-Chandra, \cite[Section 16]{hc84}) \label{pf}
For any $f \in C^{\infty}_{c}(G(F)),$
\[
f(1)  =  \dsum_{M} a(G|M) \int \limits_{\pi \in \Pi_{\disc}(M(F))} d_M (\pi) \mu_M (\pi) trace(\ii_{G,M} \pi )(f) d \pi,
\]
where the sum runs over all $F$-Levi subgroups $M$ up to $G(F)$-conjugacy.
\qed
\end{theorem}

\subsection{Unramified characters} \label{unram} 
We introduce the group $\Psi(M)$ of unramified characters of $M(F).$ 
We continue with the notation in Sections \ref{notation and convention} and \ref{def of PM}.
Let $M$ be a standard $F$-Levi subgroup of $G$ with split component $A_{M}.$ 
We denote by $Z_M$ the center of $M.$
We write 
\[ 
M(F)^{1} := \{ m \in M(F) : |\chi(m)|_F = 1 ~ \text{for all characters} ~ \chi \in  X^{*}(M)_F \}.
\]
We define the group of unramified characters $\Psi(M)$ on $M(F)$ as follows 
\[
\Psi(M):=\Hom(M(F) / M(F)^{1}, \CC^{\times}).
\]
Denote by $\Psi_u(M)$ the subgroup of $\Psi(M)$ consisting of unitary unramified characters. 
\begin{remark} \label{properties of M/M1}
We note the following facts (see \cite[Section 3.3]{mt11} and \cite[Section 1.4.1]{roc09}).
   \begin{itemize}
    \item[(i)] $M(F)^{1} = \ker(H_{M}).$
    \item[(ii)] Every compact subgroup in $M(F)$ is contained in $M(F)^{1}.$ 
     In particular, $U(F)$ is contained in $L(F)^{1}$ for all standard $F$-parabolic subgroups $Q$ of $M$ with Levi decomposition $Q=LU.$
    \item[(iii)] $M(F)^{1}$ is an open, closed, normal and unimodular subgroup in $M(F).$
    \item[(iv)] The quotient $M(F) / M(F)^{1}$ is isomorphic to $\ZZ^{d},$ where $d$ is the dimension of $A_M.$ In particular, if the center of $M$ is anisotropic over $F,$ then we have $M(F)=M(F)^{1}.$
    \item[(v)]  The index $[M(F) : Z_M(F) M(F)^1]$ is finite. Since the quotient $Z_M(F) / (Z_M(F) \cap M(F)^1)$ is a full rank sublattice of $M(F) / M(F)^1,$ the index $[M(F) / M(F)^1 : Z_M(F) / (Z_M(F) \cap M(F)^1)]$ is also finite.    
   \end{itemize}   
\end{remark}
 
\begin{remark} 
We fix a $\ZZ$-basis $(m_1, m_2, \cdots, m_d)$ for $M(F) / M(F)^{1} \s \ZZ^{d}.$ Then the map 
\[
\psi \mapsto (\psi(m_1), \psi(m_2), \cdots, \psi(m_d))
\] 
yields the following isomorphisms 
\[
\Psi(M) \s (\CC^{\times})^{d} ~ \text{and} ~  \Psi_u(M) \s (S^1)^{d},
\]
where $S^1$ is the unit circle in $\CC^{\times}$ (see \cite{roc09}).
\end{remark}
\begin{example} \label{example for unram} 
Suppose $G = M = GL_m.$ One has that $X^{*}(G)_F = \lbrace \det ^{n} : n \in \ZZ \rbrace \s \ZZ.$ 
We then have a short exact sequence
 \[
 1 \longrightarrow G^{1}(F) \longrightarrow G(F) \overset{|\det|_F}{\longrightarrow}
  q^{\ZZ} \longrightarrow 1.
 \]
We get $G(F) / G^{1}(F) \s \im(|\det|_F)=\lbrace q^{k} : k \in \ZZ \rbrace \s \ZZ.$
Thus we have $\Psi(G) \s \Hom(\ZZ, \CC^{\times}) \s \CC^{\times}$ and $\Psi_u(G) \s \Hom(\ZZ, S^{1}) \s S^{1}.$
In particular, we note that 
\[
\Psi(G) = \lbrace g \mapsto |\det(g)|^{s}_F : s \in \CC \rbrace \s \ma^{*}_{G, \CC} /  \mathcal{L}_{G},
\] 
\[\Psi_u(G) = \lbrace g \mapsto |\det(g)|^{s}_F : s \in \sqrt{-1}\RR \rbrace \s \sqrt{-1} \ma^{*}_G / \mathcal{L}_{G}.
\] 
Here $\mathcal{L}_{G}:= \frac{2 \pi \sqrt{-1}}{\log(q)} X^{*}(G)_F$ is a lattice in $\sqrt{-1} \ma^{*}_G.$
\end{example}
\begin{remark} [Epimorphism from $\ma^{*}_{M, \CC}$ onto $\Psi(M)$] \label{epimorph}
There is a natural homomorphism $\nu \mapsto \psi_{\nu}$ from 
the $\CC$-vector space $\ma^{*}_{M, \CC}$ to the group  $\Psi(M)$ of unramified characters defined by
\[
\psi_{\nu}(m) := q ^{\langle  \nu, H_{M}(m) \rangle}.
\]
We thus have the following properties (see \cite[Section 3.3]{mt11}).
\begin{itemize}
 \item[(a)]  $\psi_{\nu} \in \Psi_u(M)$ (that is, $\psi_{\nu}$ is unitary) if and only if $\nu \in \sqrt{-1} \ma^{*}_M.$
 \item[(b)]  $\psi_{\nu}$ is trivial if and only if $\nu \in \frac{2 \pi \sqrt{-1}}{\log(q)} X^{*}(M)_F \subset \sqrt{-1} \ma^{*}_M.$ Notice that $\mathcal{L}_{M} := \frac{2 \pi \sqrt{-1}}{\log(q)} X^{*}(M)_F$ is a lattice in $\sqrt{-1} \ma^{*}_M.$
 \item[(c)] The homomorphism $\nu \mapsto \psi_{\nu}$ is surjective. Hence, we have natural isomorphisms 
 \[
 \ma^{*}_{M, \CC} / \mathcal{L}_{M} \s \Psi(M) ~ \text{and} ~ \sqrt{-1} \ma^{*}_{M} / \mathcal{L}_{M} \s \Psi_u(M).
 \]
\end{itemize}
\end{remark}

\begin{remark} \label{pm at L_F}
Let $\nu \in \ma^{*}_{M, \CC},$ $\sigma \in \Irr_u(M(F))$ and $\tilde{w} \in W_G$ be given as in Definition \ref{df PM}. We observe that if $\nu \in \mathcal{L}_{M},$ then we have
$
\mu_M(\nu, \sigma, w) = \mu_M(0, \sigma, w).
$
Hence, given $\psi \in \Psi(M),$ we write $\mu_M(\sigma \otimes \psi)$ for $\mu_M(\nu, \sigma, w),$ where $\nu \in \ma^{*}_{M, \CC}$ is obtained by letting $\psi_{\nu} = \psi$ in the above natural isomorphism $\nu \mapsto \psi_{\nu}$ from $\ma^{*}_{M, \CC} / \mathcal{L}_{M}$ onto $\Psi(M).$
\end{remark}

\subsection{Regular functions on $\Psi(M)$ and $\Omega(\sigma)$} \label{reg} 
We discuss a regular function (resp., a rational function) on $\Psi(M),$ based on the notions and facts in \cite[Sections III - V]{wal03}. We continue with the notation as in Sections \ref{def of PM} and \ref{unram}. We fix a $\ZZ$-basis $(m_1, m_2, \cdots, m_d)$ for $M(F) / M(F)^{1} \s \ZZ^{d}.$ Then $\Psi(M)$ is an abelian complex algebraic group that is isomorphic to $({\CC}^{\times})^{d}$ via
\[
\lambda(\psi) = (\psi(m_1), \psi(m_2), \cdots, \psi(m_d)). 
\]
Here $d$ is the dimension of $A_M.$ Note that $\CC[z_1, z_1^{-1}, \cdots, z_d, z_d^{-1}]$ is the set of all regular functions on $({\CC}^{\times})^{d}.$  
\begin{definition} \label{def of reg ft}
A function $\xi : \Psi(M) \rightarrow \CC$ is said to be \textit{a regular function} (resp., \textit{a rational function}) on $\Psi(M)$ if there exists a regular (resp., rational) function $p(z)$ on $({\CC}^{\times})^{d}$ such that $\xi(\psi) = p(\lambda(\psi))$ for all $\psi \in \Psi(M).$ Here $z = (z_1, \cdots, z_d).$
\end{definition}
\begin{remark} (\cite[Section IV.3 and Lemma V.2.1]{wal03} and \cite[Theorem 20]{hc73}) \label{rationality of pm}
Let $\sigma \in \Irr(M(F))$ be given. The Plancherel measure $\mu_{M} (\sigma \otimes \psi)$ (see Remark \ref{pm at L_F}) is a rational function on $\Psi(M).$ 
Moreover, it is a non-negative and regular function on $\psi \in \Psi_u.$
\end{remark}
\begin{remark} \label{rem of poly}
There is a natural action of $W_M$ on the set of regular functions on $\Psi(M)$ as follows. 
Given a regular function $\xi$ on $\Psi(M),$ there exists a regular function $p(z)$ on $({\CC}^{\times})^{d}$ such that $\xi(\psi) = p(\lambda(\psi))$ for all $\psi \in \Psi(M).$ 
We define
\[
^{w}\xi (\psi) = \xi(^{w}\psi).
\]
Here, the action of the Weyl group $W_M$ on $\Psi(M)$ is defined by $^{w}\psi(m) = \psi(w^{-1}mw),$
for $w \in W_M,$ $\psi \in \Psi(M),$ and $m \in M(F).$
It then turns out that $^{w}\xi$ is again a regular function on $\Psi(M),$ since 
$
^{w}\xi(\psi) = \, ^{w}p(\lambda(\psi)),
$ where $^{w}p(z) := p(\lambda(\, ^{w}(\lambda^{-1}(z))))$ is a regular function on $({\CC}^{\times})^{d}.$ 
\end{remark}

Next, we introduce a regular function (resp., a rational function) on the $\Psi(M)$-orbit
\[
\Omega(\sigma) := \lbrace \sigma \otimes \psi : \psi \in \Psi(M)  \rbrace
\]
of $\sigma \in \Irr(M(F)).$ 
Set 
\[
{\Stab}_{\Psi}(\sigma) := \lbrace \psi \in \Psi(M) : \sigma \otimes \psi \s \sigma \rbrace 
\]
to be the stabilizer of $\sigma$ in $\Psi(M).$ 
\begin{remark} \label{finite Stab}
We note that ${\Stab}_{\Psi}(\sigma)$ is a finite subgroup in $\Psi_u(M)$ (see \cite[Section 1.4.1]{roc09}) and thus a finite subgroup in $(S^1)^d,$ where $d$ is the integer with $\Psi_u(M) \s (\CC^{\times})^d.$
There is also an isomorphism $\Omega(\sigma) \s \Psi(M) / {\Stab}_{\Psi}(\sigma),$ whence $\Omega(\sigma)$ is naturally equipped with a complex (quotient) variety structure.
\end{remark} 
\begin{definition}  \label{def of reg ft on omega} 
A function $\xi : \Omega(\sigma) \rightarrow \CC$ is said to be \textit{a regular function} (resp., \textit{a rational function}) on $\Omega(\sigma)$ if the function $\psi \mapsto \xi(\sigma \otimes \psi)$ is a regular function (resp., a rational function) on $\Psi(M).$
\end{definition}

\begin{remark} \label{rem of reg ft on omega}
For a given subgroup $H \subset \Psi(M),$ a regular (resp., rational) function $\xi$ on $\Psi(M)$ is said to be \textit{invariant under} $H$ if $\xi(\psi) = \xi(\psi')$ for all $\psi,$ $\psi'$ $\in H.$ We consider an isomorphism $\Omega(\sigma) \s \Psi(M) / {\Stab}_{\Psi}(\sigma).$ Then, given a regular (resp., rational) function $\xi_o$ on $\Psi(M)$ which is invariant under ${\Stab}_{\Psi}(\sigma),$ we set 
\[
\xi(z) := \xi_o(\psi) 
\]
for  $z = \sigma \otimes \psi \in \Omega(\sigma).$ It then turns out that $\xi$ is a regular function (resp., a rational function) on $\Omega(\sigma).$ 
\end{remark}

Here are some examples of regular functions on $\Psi(M)$ and $\Omega(\sigma).$
\begin{example}
Let $\sigma \in \Irr_u(M(F))$ be given. We denote by $m$ the cardinality of the finite group ${\Stab}_{\Psi}(\sigma).$ 
For any polynomial $p(z) \in \CC[z_1^m, z_1^{-m}, \cdots, z_d^m, z_d^{-m}],$ we set  $\xi_o(\psi)=p(\lambda(\psi))$ for all $\psi \in \Psi(M).$ 
Then $\xi_o$ is a regular function on $\Psi(M)$ that is invariant under ${\Stab}_{\Psi}(\sigma),$ since $\chi^m = 1$ for any $\chi \in {{\Stab}}_{\Psi}(\sigma).$ Hence, a function
  \[
  \xi(z) = \xi(\sigma \otimes \psi) := \xi_o(\psi)
  \]
on $\Omega(\sigma)$ is regular.
\end{example}
\begin{remark} \label{W-invariant}
Let $\xi_o$ be a regular function on $\Psi(M).$ For any $w \in W_M,$ Remark \ref{rem of poly} implies that $^{w}\xi_o$ is again a regular function on $\Psi(M).$ Define
\begin{equation} \label{W invariant equation}
\xi(z) := \dfrac{1}{|W_M|}\dsum_{w \in W_M} \, ^{w}\xi_o(z).
\end{equation}
Then $\xi(z)$ turns out to be a $W_M$-invariant regular function on $\Psi(M),$ that is, $^{w}\xi = \xi$ for any $w \in W_M.$ 
Therefore, there is always a $W_M$-invariant regular function on $\Psi(M).$
In fact, this holds for any subgroup of $W_M$ by replacing $W_M$ in \eqref{W invariant equation} with the subgroup. 
\end{remark}

\section{Generalized pseudo-coefficients} \label{sec2}
We continue with the notation of Section \ref{prelim}.
Let $G$ be a connected reductive algebraic group over a $p$-adic field $F$ of characteristic $0,$ and let $M$ be a standard $F$-Levi subgroup (possibly, $M=G$) of $G$ with split component $A_{M}.$ 
In this section, for any $\sigma \in \Pi_{\disc}(M(F))$ and regular function $\xi$ on the $\Psi(M)$-orbit $\Omega(\sigma)$ of $\sigma,$ we obtain a function $\phi_{\sigma, \xi} \in C^{\infty}_{c}(M(F))$ satisfying a certain trace relation (Theorem \ref{phi}). 
We then explore several properties of $\phi_{\sigma, \xi}$ and lift $\phi_{\sigma, \xi} \in C^{\infty}_{c}(M(F))$ to $f_{\sigma, \xi} \in C^{\infty}_{c}(G(F)).$ Finally, we apply the Plancherel formula to $\phi_{\sigma, \xi}$ and $f_{\sigma, \xi}.$

Following the notation of \cite{bdk}, 
we denote by $R(M)$ the Grothendieck group of smooth representations of $M(F)$ ($M(F)$-module) of finite length. Note that $R(M)$ is a free  abelian group with basis $\Irr(M(F)).$ 
For a standard $F$-parabolic subgroup $Q=LU$ of $M,$ the normalized (twisted by $\delta_{Q}^{1/2}$) induced representation $\ii_{M,L} (\pi)$ from $\pi \in \Irr(L(F))$ defines the morphism $\ii_{M,L} : R(L) \rightarrow R(M).$ 
Here $\delta_{Q}$ is the modulus character of $Q(F).$
A cuspidal pair is defined by ($L,$ $\rho$), where $L$ is a standard $F$-Levi subgroup of $M$ and $\rho \in \Irr(L(F))$ is an irreducible supercuspidal $L(F)$-module. 
We denote by $\Theta(M)$ the set of all cuspidal pairs up to $M(F)$-conjugation.
We say ($L,$ $\rho$) and ($L',$ $\rho'$) are $M(F)$-conjugate if there exists $m \in M(F)$ 
  such that $L' = mLm^{-1}$ and $\rho' = m \rho m^{-1}.$ 
  
Given a cuspidal pair ($L,$ $\rho$), the set $\lbrace (L, \psi \rho) \vert \psi \in \Psi(L) \rbrace$ is called a connected component of $\Theta(M).$ For each $\pi \in \Irr(M(F)),$ there exists a cuspidal pair $(L, \rho)$ unique up to conjugation by $M$ such that $\pi$ is a sub-quotient of $\ii_{M,L}(\rho).$ By sending $\pi$ to $(L, \rho),$ we define a map $\ic : \Irr(M(F)) \rightarrow \Theta(M).$ The map $\ic$ is $\Psi(M)$-equivariant, onto, and finite-to-one. 
Given a connected component $\Theta$ of $\Theta(M),$ the inverse image $\ic^{-1}(\Theta)$ is called a connected component of $\Irr(M(F)).$ 

\subsection{Construction of a generalized pseudo-coefficient} \label{construction} 
We recall the trace Paley-Wiener theorem from \cite[Section 1.2]{bdk} and \cite[Section 4.2]{clo86}.

\begin{definition} \label{good}
A linear functional (group homomorphism) $\lambda : R(M) \rightarrow \CC$ is called \textit{good} if
  \begin{itemize}  
   \item [(i)] the function $\lambda: \Irr(M(F)) \rightarrow \CC$ is supported on a finite number of components;
   \item [(ii)] for any standard $F$-Levi subgroup $L$ and $\pi \in R(M),$ the function $\psi \mapsto \lambda(\ii_{M,L}(\pi \otimes \psi))$ is a regular function on the complex variety $\Psi(L).$
  \end{itemize} 
\end{definition}

\begin{proposition}[Bernstein-Deligne-Kazhdan] \label{pw}
Let $\lambda$ be as in Definition \ref{good}. $\lambda$ is \textit{good} if and only if there is a function $h \in C^{\infty}_{c}(M(F))$ such that for any $\pi \in R(M)$
\[
\lambda(\pi) = trace(\pi)(h).
\]
\end{proposition}

\begin{remark}
Instead of (i) in Definition \ref{good}, it is customary to use the condition that there exists an open compact subgroup $K_M \subset M(F)$ such that $\lambda$ is non-zero only on the $M(F)$-modules having nontrivial $K_M$-invariant vectors. We note that both are equivalent by Decomposition Theorem in \cite[Section 2.3]{bdk}.
\end{remark}

\begin{definition}
We say that the normalized induced representation $\ii_{M,L}(\pi)$ is \textit{a standard module} if $\pi \s \tau \otimes q ^{\langle  \nu, H_{L}( \,) \rangle}$ for some $\tau \in \Pi_{\temp}(L(F))$ and $\nu \in (\ma^{*}_{L})^{+}:= \lbrace \nu \in \ma^{*}_{L} : \langle \alpha^{\vee}, \nu \rangle > 0, \; \; \forall \alpha \in \Delta_M \smallsetminus \Delta_L \rbrace$ (the positive Weyl chamber), where $\Delta_M$ and $\Delta_L$ are the sets of simple roots in $M$ and $L$ respectively (see \cite[p.279]{clo86} where it is called a standard character).
\end{definition}
\begin{proposition}(\cite[Proposition 2 p.279 ]{clo86}) \label{standard}
The standard modules form a basis over $\ZZ$ for the Grothendieck group $R(M).$
\end{proposition}
\begin{theorem}[Existence of generalized pseudo-coefficients] \label{phi}
For any $\sigma \in \Pi_{\disc}(M(F))$ and any regular function $\xi$ on $\Omega(\sigma),$ there exists a function $\phi_{\sigma, \xi} \in C^{\infty}_{c}(M(F))$ such that for all $\pi \in \Pi_{\temp}(M(F)),$
\[
  trace(\pi)(\phi_{\sigma, \xi}) = \left\{ 
  \begin{array}{l l}
    \xi(\pi), & \: \text{if} \quad \pi \s \sigma \otimes \psi \: \text{for some} \: \psi \in  \Psi(M), \\

    0, & \: \text{otherwise.}\\
  \end{array} \right.
\]
\end{theorem}
We call the function $\phi_{\sigma, \xi} \in C^{\infty}_{c}(M(F))$ \textit{a generalized pseudo-coefficient for $\sigma$ and $\xi$}.

\begin{remark}[Uniqueness of generalized pseudo-coefficients] \label{uniquely determined}
For given a regular function $\xi$ on $\Omega(\sigma),$ the function $\phi_{\sigma, \xi}$ is determined \textit{uniquely modulo the $\CC$-subspace $J_M$} spanned by the functions of the form $m \mapsto (h(m)-h(xmx^{-1}))$ for all $h \in C^{\infty}_{c}(M(F))$ and $x \in M(F).$ This follows from the fact (\cite[Theorem 0]{kaz}) that 
\[
\phi_1 - \phi_2 \in J_M \; \Longleftrightarrow \; trace (\pi)(\phi_1)=trace (\pi)(\phi_2), \quad \forall \pi \in \Pi_{\temp}(M(F)) ~
.
\]
\end{remark}
\begin{remark} 
We note that, since $\pi \in \Pi_{\temp}(M(F))$ is unitary, the unramified character $\psi$ satisfying $\pi \s \sigma \otimes \psi$ must be unitary, that is, in $\Psi_u(M).$
\end{remark}

\begin{proof}[Proof of Theorem \ref{phi}]
We refer to Clozel's proof of \cite[Proposition 1]{clo86}. Given a regular function $\xi$ on $\Omega(\sigma),$ we define a linear functional $\lambda : R(M) \rightarrow \CC$ such that 
\[
  \lambda(\pi) = \left\{ 
  \begin{array}{l l}
    \xi(\pi), & \: \text{if} \quad \pi \s \sigma \otimes \psi \: \text{for some} \: \psi \in  \Psi(M), \\

    0, & \: \text{if} \quad \pi \: \text{is a standard module but not in} \: \: \Omega(\sigma). \\
  \end{array} \right.
\]
Since essentially square-integrable representations can never be realized by an induced representation from a proper $F$-Levi subgroup and by Proposition \ref{standard}, the linear functional $\lambda$ is well-defined.

Now we show that $\lambda$ is \textit{good} (see Proposition \ref{pw}). If the image $\ic(\sigma)$ is $(L, \rho)$ for some standard $F$-Levi subgroup $L$ of $M$ and a supercuspidal representation $\rho$ of $L(F),$ then $\ic$ maps $\Omega(\sigma)$ into the connected component $\Theta$ containing $(L, \rho).$ Indeed, since the map $\ic$  is $\Psi(M)$-equivariant, we have $\ic(\Psi(M) \sigma) = (L, \Psi(M) \rho) \subset (L, \Psi(L) \rho)= \Theta.$ It follows that the linear functional is supported only on $\Theta,$ so the condition (i) of Definition \ref{good} is fulfilled. 

For the condition (ii) of Definition \ref{good}, it suffices to show that $\lambda$ vanishes on any induced representation $\ii_{M,L}(\tau)$ with proper standard $F$-Levi subgroup $L$ of $M$ and $\tau \in R(L).$ By Proposition \ref{standard}, we may assume that $\tau$ is a standard module $\ii_{L,L_1} (\varsigma),$ where $L_1$ is a standard $F$-Levi subgroup of $L$ and $\varsigma \s \varsigma_1 \otimes q ^{\langle  \nu_1, H_{L_1}( \,) \rangle}$ for some $\varsigma_1 \in \Pi_{\temp}(L_1(F))$ and $\nu_1 \in (\ma^{*}_{L_1})^{+}.$ Then $\ii_{M,L}(\tau)$ is isomorphic to $\ii_{M,L_1} (\varsigma)$ which is a standard module but not an essentially square-integrable representation (since $L_1$ is proper), that is, $\ii_{M,L}(\tau) \notin \Omega(\sigma).$ It turns out that $\lambda(\ii_{M,L}(\tau))=0$ for any proper standard $F$-Levi subgroup $L$ of $M$ and any $\tau \in R(L).$ Hence, the condition (ii) also follows, and $\lambda$ is \textit{good}. Due to the trace Paley-Wiener theorem (Proposition \ref{pw}), the proof is complete.
\end{proof}

\subsection{Properties of $\phi_{\sigma, \xi}$} \label{properties}
We provide several properties of the generalized pseudo-coefficient $\phi_{\sigma, \xi}$ for $\sigma$ and $\xi$ defined in Theorem \ref{phi}.  
Denote by $\omega$ a unitary character of $Z_M(F),$ where $Z_M$ is the center of $M.$ 
By $C^{\infty}_{c}(M(F), \omega^{-1})$ we denote the Hecke algebra of locally constant functions with compact support modulo the center and transforming under the center by $\omega^{-1}.$ Let $\Pi_{\temp}(M(F), \omega)$ and $\Pi_{\disc}(M(F), \omega)$ be the sets of equivalent classes of irreducible tempered representations and discrete series representations of $M(F),$ respectively, whose central character is $\omega.$

Let $\sigma \in \Pi_{\disc}(M(F), \omega)$ be given.
The projection $\phi_{\sigma, \xi}^{\omega}$ of $\phi_{\sigma, \xi} \in C^{\infty}_{c}(M(F))$ onto $C^{\infty}_{c}(M(F), \omega^{-1})$ is defined by
\begin{equation} \label{projection}
\phi_{\sigma, \xi}^{\omega}(m) := \int _{Z_M(F)} \omega(z) \phi_{\sigma, \xi}(zm) dz, 
\end{equation}
where $dz$ is a Haar measure on $Z_M(F).$  

Given $h \in C^{\infty}_{c}(M(F)),$ the orbital integral $\mathit{O}^{M(F)}_{\gamma}(h)$ of $h$ for $\gamma \in M(F)$ is defined by
\begin{equation} \label{def orbital int}
\mathit{O}^{M(F)}_{\gamma}(h):= \int _{M_{\gamma}^\circ(F) \backslash M(F)} h(x^{-1} \gamma x) d x,
\end{equation} 
where $d x$ is a quotient Haar measure on the quotient $M_{\gamma}^\circ(F)\backslash M(F).$ 
Given a character $\omega$ of $Z_M(F),$ we define $\mathit{O}^{M(F)}_{\gamma}(h^{\omega})$ for $h^{\omega} \in C^{\infty}_{c}(M(F), \omega^{-1})$ in the same way.

We present the cuspidal properties of $\phi_{\sigma, \xi}$ and $\phi^{\omega}_{\sigma, \xi}.$
\begin{proposition}[Cuspidality of $\phi_{\sigma, \xi}$ and $\phi^{\omega}_{\sigma, \xi}$] \label{cuspidal property}
With the function $\phi_{\sigma, \xi} \in C^{\infty}_{c}(M(F))$ in Theorem \ref{phi} and its projection $\phi_{\sigma, \xi}^{\omega}$ defined in \eqref{projection}, 
we have
\[
\mathit{O}^{M(F)}_{\gamma}(\phi_{\sigma, \xi}) = 0 = \mathit{O}^{M(F)}_{\gamma}(\phi_{\sigma, \xi}^{\omega})
\]
for any $\gamma \in M(F)^{\reg} \smallsetminus M(F)^{\el}.$
\end{proposition}
\begin{proof}
We first notice from Definition \eqref{projection} that 
\[
\mathit{O}^{M(F)}_{\gamma}(\phi_{\sigma, \xi}^{\omega}) = \int_{Z_M(F)} \omega(z) \mathit{O}^{M(F)}_{z \gamma}(\phi_{\sigma, \xi})dz.
\]
It is sufficient to show that for every $\gamma \in M(F)^{\reg} \smallsetminus M(F)^{\el},$
\[
\mathit{O}^{M(F)}_{\gamma}(\phi_{\sigma, \xi})=0.
\]
Let $\gamma \in M(F)^{\reg} \smallsetminus  M(F)^{\el}$ be given. Then there exists a proper standard $F$-Levi subgroup $L$ of $M$ containing $\gamma.$ 
We denote by $Q$ the standard $F$-parabolic subgroup in $M$ with its Levi factor $L.$ 
We consider the following parabolic descent $\bar{\phi}_{\sigma, \xi}^{(Q)}$ of $\phi_{\sigma, \xi}$ which is due to Harish-Chandra (cf. \cite[p.236]{dij72} and \cite[Section 13]{kot05}). By \cite[Lemma 9]{dij72}, we have 
\begin{equation} \label{orbital int for par descent}
|D_{M/L}(\gamma)| ^{1/2} \mathit{O}^{M(F)}_{\gamma}(\phi_{\sigma, \xi}) 
= \mathit{O}^{L(F)}_{\gamma}(\bar{\phi}_{\sigma, \xi}^{(Q)}),
\end{equation}
where $D_{M/L}(\gamma) = \det(1-ad(\gamma)) |_{\Lie(M)/\Lie(L)}.$ We claim that 
\[
\mathit{O}^{L(F)}_{\gamma}(\bar{\phi}_{\sigma, \xi}^{(Q)})=0
\]
for all $\gamma \in L(F).$
It is well-known (cf. \cite[p.237]{dij72} and \cite[Section 13.7]{kot05}) that 
\[
trace(\pi)(\bar{\phi}_{\sigma, \xi}^{(Q)})
=trace(\ii_{M,L}(\pi))(\phi_{\sigma, \xi})
\] 
for any $\pi \in \Irr(L(F)).$
Since $L$ is proper and the orbit $\Omega(\sigma)$ consists only of essentially square-integrable representations, $\ii_{M,L}(\pi)$ cannot be in $\Omega(\sigma).$ Theorem \ref{phi} implies that 
\begin{equation} \label{trace of phi vanishing}
trace(\ii_{M,L}(\pi))(\phi_{\sigma, \xi})=0
\end{equation}
for any $\pi \in \Irr(L(F)).$ It follows that $trace(\pi)(\bar{\phi}_{\sigma, \xi}^{(Q)})=0.$ Therefore, the claim is true from the fact (\cite[Theorem 0]{kaz}) that 
\[
trace(\pi)(\bar{\phi}_{\sigma, \xi}^{(Q)})=0 \quad \text{for all} ~ \pi \in \Pi_{\temp}(L(F))
\] 
if and only if 
\[ 
\mathit{O}^{L(F)}_{\gamma}(\bar{\phi}_{\sigma, \xi}^{(Q)})=0 \quad \text{for all}~ \gamma \in M(F)^{\reg}.
\] 
Since $D_{M/L}$ is never zero for any element in $M(F)^{\reg},$ the proof is complete.
\end{proof}

\begin{remark}
Proposition \ref{cuspidal property} is called the Selberg principle (cf. \cite[Lemma 45]{hc162}). 
Furthermore, it turns out that $\phi_{\sigma, \xi}$ is a cuspidal function in the sense of Arthur (see \cite[p.130]{art93}).
\end{remark}

We give a relationship between the orbital integral of  $\phi_{\sigma, \xi}^{\omega}$ and the Harish-Chandra character $\Theta_{\sigma}$ (defined below) over the elliptic regular semi-simple set $M(F)^{\el}.$ 

\begin{definition} (Harish-Chandra, \cite[p.96]{hc81}) \label{hc character}
For any $\pi \in \Irr(M(F)),$ there exists a unique locally constant function on $M(F)^{\reg}$ which is invariant under conjugation by $M(F)$ such that for all $ h \in C^{\infty}_{c}(M(F)),$
\[
trace (\pi) (h) = \int_{M(F)^{\reg}} h(m) \Theta_{\pi}(m) dm,
\]
with $dm$  a Haar measure on $M(F)^{\reg}.$ We call $\Theta_{\pi},$ \textit{the Harish-Chandra character} of $\sigma.$
\end{definition}
We note that the distribution character $trace (\pi)$ is dependent upon the choice of Haar measure $dm,$ but the Harish-Chandra character $\Theta_{\pi}$ is not.
Furthermore, Definition \ref{hc character} stems from a theorem of Harish-Chandra, which proves the existence and uniqueness (see \cite{hculs} for details).

We present the following lemma to state Proposition \ref{pro for relation bw character and orb int}.
\begin{lemma} \label{some lemma for finiteness}
For any $\sigma \in \Pi_{\disc}(M(F), \omega),$ the set $\Omega(\sigma) \cap \Pi_{\disc}(M(F), \omega)$ is finite.
\end{lemma}
\begin{proof}
Let $\pi \in \Pi_{\disc}(M(F), \omega)$ be given such that $\pi \s \sigma \otimes \psi$ for some $\psi \in \Psi(M).$ Since both central characters of $\sigma$ and $\pi$ are the same as $\omega,$ the restriction $\psi|_{Z_M(F)}$ of $\psi$ to $Z_M(F)$ is trivial. 
The set $\Omega(\sigma) \cap \Pi_{\disc}(M(F), \omega)$ is a subset of
\begin{equation} \label{some finite set}
\lbrace
\psi \in \Psi_u(M) : \psi | _{Z_M(F)} = \mathbbm{1}
\rbrace.
\end{equation}
Note that $\psi$ is already trivial on $M(F)^1.$ Since the index $[M(F) : Z_M(F) M(F)^1]$ is finite from Remark \ref{properties of M/M1}, there are only finitely many characters on $M(F)$ whose restriction to $Z_M(F)M(F)^1$ is trivial. Thus the set \eqref{some finite set} is finite, which implies that the set $\Omega(\sigma) \cap \Pi_{\disc}(M(F), \omega)$ is finite.
\end{proof}

\begin{proposition} [Relationship between $\phi_{\sigma, \xi}$ and $\Theta_{\sigma}$] \label{pro for relation bw character and orb int} 
Let $\sigma \in \Pi_{\disc}(M(F), \omega)$ be given, and let $\xi$ be a regular function on $\Omega(\sigma).$ Suppose that $\xi(z)$ vanishes on $\Omega(\sigma) \cap \Pi_{\disc}(M(F), \omega)$ unless $z  \s \sigma.$ Then we have
\begin{equation} \label{equ for ch and or}
\mathit{O}^{M(F)}_{\gamma}(\phi_{\sigma, \xi}^{\omega}) = \xi(\sigma) \overline{\Theta_{\sigma}(\gamma)}
\end{equation}
for every $\gamma \in M(F)^{\el}.$ Here the bar $\, \bar{~~~} \,$ means the complex conjugation.
\end{proposition}
\begin{remark}
The orbital integral in the left-hand side is computed with respect to the Haar measure $dx$ in Definition \eqref{def orbital int}. Changing $dx$ to $c \, dx$ for some non-zero constant $c,$ we then get $c^{-1} \, \phi_{\sigma, \xi}$ from Proposition \ref{phi} in place of $\phi_{\sigma, \xi}.$ Thus the left-hand side is independent of choice of Haar measure, as the right-hand side should be. Note that the right-hand side depends only on $\xi$ and $\sigma.$
\end{remark}
\begin{proof}[Proof of Proposition \ref{pro for relation bw character and orb int}]
For any matrix coefficient $h^{\omega}_{\sigma}$ of $\sigma,$ from the proof of \cite[Proposition 6]{clo89}, there exists $\widetilde{h^{\omega}_{\sigma}} \in C^{\infty}_{c}(M(F), \omega^{-1})$ such that
\begin{equation} \label{two orbital integrals}
\mathit{O}^{M(F)}_{\gamma}(\widetilde{h^{\omega}_{\sigma}})
= \mathit{O}^{M(F)}_{\gamma}(h^{\omega}_{\sigma}) 
, \quad \forall \gamma \in M(F).
\end{equation}
We recall the Weyl integration formula (cf. \cite[A.3.f]{dkv} and \cite[Section 7]{kot05}).
For any function $h^{\omega} \in C^{\infty}_{c}(M(F), \omega^{-1})$ and any $\pi \in \Pi_{\temp}(M(F), \omega),$
\[
trace(\pi)(h^{\omega})= \dsum_{T} |W_T|^{-1} \int_{Z_M(F) \backslash T(F)^{\reg}} 
|D_M(t)| \mathit{O}^{M(F)}_{\gamma}(h^{w}) \Theta_{\pi}(t) dt,
\]
where the sum is taken over $M(F)$-conjugacy classes of maximal $F$-tori $T$ in $M,$  $W_T:=N_M(T) / Z_M(T),$ $T(F)^{\reg} := M(F)^{\reg} \cap T(F),$ $D_M(t):=\det(1-ad(t)) |_{\Lie(M)/\Lie(T)},$ and $dt$ is a Haar measure on $Z_M(F) \backslash T(F)^{\reg}.$
From \eqref{two orbital integrals} and the Weyl integration formula, we deduce that
\begin{equation} \label{two traces}
  trace(\pi)(\widetilde{h^{\omega}_{\sigma}}) 
  = trace(\pi)(h^{\omega}_{\sigma})
\end{equation} 
for any $\pi \in \Pi_{\temp}(M(F), \omega).$ Moreover, it is well-known [DKV84, Proposition A.3.g] that 
\begin{equation} \label{trace values}
trace(\pi)(h^{\omega}_{\sigma}) =
  \left\{ 
  \begin{array}{l l}
    \dfrac{h^{\omega}_{\sigma}(1)}{d_M(\sigma)}, & \: \text{if} \, \, \pi  \s  \sigma, \\

    \quad 0, & \: \text{otherwise.}\\
  \end{array} \right.
\end{equation} 
Here $d_M(\sigma)$ is the formal degree. We also recall from \cite[Proposition A.3.e.2)]{dkv} that for all $\gamma \in M(F)^{\reg},$ we have
\begin{equation} \label{matrix coeff and character}
\mathit{O}^{M(F)}_{\gamma}(h^{\omega}_{\sigma}) =
  \left\{ 
  \begin{array}{l l}
    \dfrac{h^{\omega}_{\sigma}(1)}{d_M(\sigma)} \; \overline{\Theta_{\sigma}(\gamma)},
     & \: \text{if} \, \, \gamma \in M(F)^{\el}, \\

    \quad 0, & \: \text{if} \, \, \gamma \in M(F)^{\reg} \smallsetminus M(F)^{\el}.\\
  \end{array} \right.
\end{equation}
For any given regular function $\xi$ on $\Omega(\sigma),$ we choose a matrix coefficient $h^{\omega}_{\sigma}$ of $\sigma$ such that 
\begin{equation} \label{value at 1}
h^{\omega}_{\sigma}(1)=\xi(\sigma) d_M(\sigma).
\end{equation}
It then follows from \eqref{two orbital integrals} and \eqref{matrix coeff and character} that
\[
\mathit{O}^{M(F)}_{\gamma}(\widetilde{h^{\omega}_{\sigma}}) = \xi(\sigma) \overline{\Theta_{\sigma}(\gamma)}
\]
for all $\gamma \in M(F)^{\el}.$ We also have the following.
\begin{lemma} \label{lemma}
For any $\pi \in \Pi_{\temp}(M(F), \omega),$ 
\[
trace(\pi)(\phi_{\sigma, \xi}^{\omega}) = trace(\pi)(\widetilde{h^{\omega}_{\sigma}}).
\]
\end{lemma}
\begin{proof}[Proof of Lemma \ref{lemma}]
Since $\xi(z)$ vanishes on $\Omega(\sigma) \cap \Pi_{\disc}(M(F), \omega)$ unless $z \s \sigma,$ one sees from Theorem \ref{phi} that for any $\pi \in \Pi_{\temp}(M(F), \omega),$
\begin{equation} \label{trace with special xi}
  trace(\pi)(\phi_{\sigma, \xi}^{\omega})  = \left\{ 
  \begin{array}{l l}
    \xi(\sigma), & \: \text{if} \, \, \pi  \s  \sigma, \\

    \quad 0, & \: \text{otherwise.}\\
  \end{array} \right.
\end{equation}
Therefore, since $h^{\omega}_{\sigma}(1)=\xi(\sigma) d_M(\sigma),$ the lemma follows from  \eqref{two traces} and \eqref{trace values}.
\end{proof}
Continuing with the proof of Proposition \ref{pro for relation bw character and orb int}, for given $h^{\omega} \in C^{\infty}_{c}(M(F), \omega^{-1}),$ it is well-known that 
\[
trace(\pi)(h^{\omega})=0, ~ \forall \pi \in \Pi_{\temp}(M(F), \omega) 
~
\Longrightarrow
~ 
\mathit{O}^{M(F)}_{\gamma}(h^{\omega})=0,  ~ \forall \gamma \in M(F)^{\reg}
\]
(see \cite[A.2.b]{dkv}). 
From Lemma \ref{lemma} we have 
\[
\mathit{O}^{M(F)}_{\gamma}(\phi_{\sigma, \xi}^{\omega}) = 
\mathit{O}^{M(F)}_{\gamma}(\widetilde{h^{\omega}_{\sigma}})
\]
for all $\gamma \in M(F)^{\reg}.$ Hence, the proof of Proposition \ref{pro for relation bw character and orb int} is complete.
\end{proof}
\begin{remark} \label{rem for the value}
From \eqref{value at 1}, we note that the value $\xi(\sigma)$ in \eqref{equ for ch and or} must be determined by both $h^{\omega}_{\sigma}(1)$ and $d_M(\sigma).$
\end{remark}
We provide a relationship between a generalized pseudo-coefficient and a pseudo-coefficient of $\sigma.$
\begin{definition} (\cite[A.4]{dkv})
For $\delta \in \Pi_{\disc}(M(F),\omega),$ we say a function $\phi^{\omega}_{\delta} \in C^{\infty}_{c}(M(F), \omega^{-1})$ is \textit{a pseudo-coefficient} of $\delta$ if for $\pi \in \Pi_{\temp}(M(F), \omega),$
\[
  trace(\pi)(\phi^{\omega}_{\delta}) = \left\{ 
  \begin{array}{l l}
    1, & \: \text{if} \, \, \pi  \s  \delta, \\

    0, & \: \text{otherwise.}\\
  \end{array} \right.
\]
\end{definition} 

\begin{proposition} (\cite[Propositions 5(ii) and 6]{clo89})  \label{clozel}
For $\delta \in \Pi_{\disc}(M(F), \omega),$ there exists a pseudo-coefficient $\phi^{\omega}_{\delta} \in C^{\infty}_{c}(M(F), \omega^{-1})$ satisfying:
\[
\mathit{O}^{M(F)}_{\gamma}(\phi^{\omega}_{\delta}) = \overline{\Theta_{\delta}(\gamma)}, \quad \forall \gamma \in M(F)^{\el}.
\]
\end{proposition}
\begin{proof}
By letting the matrix coefficient $\phi^{\omega}_{\delta}(1)= d_M(\delta)$ (the formal degree of $\delta$) in the proof of \cite[Proposition 6]{clo89}, we have a pseudo-coefficient $\phi^{\omega}_{\delta}$ whose orbital integral is the same as that of the matrix coefficient $h^{\omega}_{\delta}$ in the proof of Proposition \ref{pro for relation bw character and orb int}. 
Then, \cite[Proposition 5(ii)]{clo89} implies Proposition \ref{clozel}.
\end{proof}

\begin{proposition} \label{pseudo}
Let $\xi$ be as in Proposition \ref{pro for relation bw character and orb int}. Suppose $\xi(\sigma) = 1.$ Then, the projection $\phi_{\sigma, \xi}^{\omega} \in C^{\infty}_{c}(M(F), \omega^{-1})$ is a pseudo-coefficient of $\sigma.$
\end{proposition}
\begin{proof}
This is a consequence of \eqref{trace with special xi} with $\xi(\sigma)=1.$
\end{proof}

\subsection{Lifting $\phi_{\sigma, \xi} \in C^{\infty}_{c}(M(F))$ to $f_{\sigma, \xi} \in C^{\infty}_{c}(G(F))$} \label{lifting}
We lift the generalized pseudo-coefficient $\phi_{\sigma, \xi} \in C^{\infty}_{c}(M(F))$ to $f_{\sigma, \xi} \in C^{\infty}_{c}(G(F))$ and discuss a relationship between $\phi_{\sigma, \xi}$ and $f_{\sigma, \xi}.$ 
Following the notation of \cite{bdk}, given a standard $F$-parabolic subgroup $P=MN$ of $G,$ 
the normalized (twisted by $\delta_{P}^{-1/2}$) Jacquet functor $\rr_{M,G} (\Pi)$ for $\Pi \in \Irr(G(F))$ defines the morphism $\rr_{M,G} : R(G) \rightarrow R(M).$
Let $R^*(G)$ (resp., $R^*(M)$) be the space of all linear functionals on $R(G)$ (resp., $R(M)$).
By $\ii^*_{G,M} : R^*(G) \rightarrow R^*(M)$ and $\rr^*_{M,G} : R^*(M) \rightarrow R^*(G)$ we denote morphisms adjoint to $\ii_{G,M}$ and $\rr_{M,G},$ respectively. 
Namely, $\ii^*_{G,M}(\mathsf{f}) = \mathsf{f}(\ii_{G,M}(\pi))$ for $\mathsf{f} \in R^*(G)$ and $\pi \in R(M),$ and $\rr^*_{M,G}(\mathsf{h}) = \mathsf{h}(\rr_{M,G}(\Pi))$ for $\mathsf{h} \in R^*(M)$ and $\Pi \in R(G).$
We set $\mathcal{F}_{tr}(G) = \lbrace trace(\Pi)(f) \; | \;  \Pi \in R(G), \; f \in C^{\infty}_{c}(G(F))\rbrace$ and $\mathcal{F}_{tr}(M) = \lbrace trace(\pi)(h) \; | \;  \pi \in R(M), \; h \in C^{\infty}_{c}(M(F))\rbrace.$
\begin{proposition} \label{existence and uniqueness of f}
For any generalized pseudo-coefficient $\phi_{\sigma, \xi} \in C^{\infty}_{c}(M(F)),$ there exists $f_{\sigma, \xi} \in C^{\infty}_{c}(G(F))$ such that
\begin{equation} \label{f_phi}
trace(\Pi)(f_{\sigma, \xi}) = trace (\rr_{M,G} ( \Pi ))(\phi_{\sigma, \xi}), 
\quad \quad  \forall \Pi \in R(G).  
\end{equation} 
Moreover, the function $f_{\sigma, \xi} \in C^{\infty}_{c}(G(F))$ is uniquely determined modulo $J_G$ (cf. Remark \ref{uniquely determined}).
\end{proposition}
\begin{proof}
The first part is a consequence of the fact (\cite[Proposition 3.2(ii)]{bdk}) that
\[
\rr^*_{M,G} (\mathcal{F}_{tr}(M)) \subset \mathcal{F}_{tr}(G). 
\] 
The uniqueness is due to the following short exact sequence (\cite[Remark 1.2]{bdk} and \cite[Theorem 0]{kaz})
\[
0 \longrightarrow J_G \longrightarrow C^{\infty}_{c}(G(F)) \xrightarrow{trace} \mathcal{F}_{tr}(G) \longrightarrow 0,
\]
where $trace$ maps $f \mapsto (\Pi \mapsto trace(\Pi)(f))$ from $C^{\infty}_{c}(G(F))$ onto $\mathcal{F}_{tr}(G).$
\end{proof}
 
For any standard $F$-Levi subgroups $M$ and $L,$ we set 
\[
W(L, M) :=\lbrace w \in W_G : wLw^{-1} = M \rbrace.
\]
We say that $L$ and $M$ are \textit{associate} if $W(L, M)$ is non-empty. When $L$ and $M$ are associate, it turns out that $W_M \cdot W(L,M) \cdot W_L = W(L, M)$ since the element $w_M w w_L$ sends $L$ to $M$ for $w_M \in W_M,$ $w \in W(L,M)$ and $w_L \in W_L.$ Furthermore, given any standard $F$-Levi subgroups $L$ and $M$ corresponding to $\vartheta$ and $\theta,$ respectively, we note from \cite[pp.448-449]{bz77} that 
\begin{equation} \label{an equal}
W_M  \backslash W(L, M) / W_L = W(L, M) / W_L \s W(\vartheta, \theta),
\end{equation}
where $W(\vartheta, \theta):= \lbrace w \in W_G : w(\vartheta) = \theta \rbrace.$ 
For any $\pi \in \Pi_{\temp}(L(F))$ and $w \in W(L, M),$ we denote by $^w{\pi}$ the representation of $M(F)$ such that $^w{\pi}(m) :=\pi(w^{-1}mw)$ for $m \in M(F).$ 
It is clear that $^w{\pi} \in  \Pi_{\temp}(M(F)).$
\begin{proposition} \label{trace of f_phi and phi}
For any standard $F$-Levi subgroup $L$ of $G$ and $\pi \in \Pi_{\temp}(L(F)),$
\[
trace(\ii_{G,L} \pi )(f_{\sigma, \xi})
= \left\{ 
  \begin{array}{l l}
     \sum_{w \in W(\vartheta, \theta)} trace( ^w{\pi})(\phi_{\sigma, \xi}),
        & \quad \text{if} \; \;  L \; \text{and} \; M \; \text{are associate}, \\
        
     \quad \; \; 0,  
        & \quad \text{otherwise.}
  \end{array} \right.
\]
\end{proposition}
\begin{proof}
We first recall from \cite[Lemma 2.11(a) and 2.12]{bz77} and \cite[Lemma 5.4(ii)]{bdk} that
\[
\rr_{M,G} \circ \ii_{G,L} (\pi) = \dsum_{w \in W_M \backslash W_G / W_L} \ii_{M,M_w} \circ w \circ \rr_{L_w, L} (\pi),
\]
where $M_w = M \cap wLw^{-1}$ and $L_w=w^{-1}Mw \cap L.$ Then \eqref{f_phi} says that for any $\pi \in \Irr(L(F)),$
\begin{align*}
trace(\ii_{G,L} \pi )(f_{\sigma, \xi}) 
&= trace(\rr_{M,G} \circ \ii_{G,L} (\pi))(\phi_{\sigma, \xi})& \\
&= \dsum_{w \in W_M \backslash W_G / W_L} trace(\ii_{M,M_w} \circ w \circ \rr_{L_w, L} (\pi) )(\phi_{\sigma, \xi}).&
\end{align*}
By \eqref{trace of phi vanishing}, it turns out that the sum runs over $w \in W(L, M) / W_L$ only when $L$ and $M$ are associate. Otherwise, the group $M_w = M \cap wLw^{-1}$ has to be proper, so the sum vanishes by the cuspidal property of $\phi_{\sigma, \xi}$ (see Proposition \ref{cuspidal property}). When $L$ and $M$ are associate, we note that $\ii_{M,M_w} \circ w \circ \rr_{L_w, L} (\pi) = \, ^w{\pi}$ and $W(L, M) / W_L \s W(\vartheta, \theta).$ Thus the proof is complete.
\end{proof}

Now we consider the case that $L = M,$ that is, $\vartheta = \theta.$ Write $W(\theta, \theta):=\lbrace w \in W_G : w(\theta) = \theta \rbrace.$ We say that $M$ is \textit{self-associate} if $W(\theta, \theta) \supsetneqq \lbrace 1 \rbrace.$ Otherwise, that is, if $W(\theta, \theta) = \lbrace 1 \rbrace,$ $M$ is called \textit{non self-associate}. We note that
\begin{equation} \label{equality for pro below}
trace(\ii_{G,M} \pi )(f_{\sigma, \xi})=\sum_{w \in W(\theta, \theta)} \, ^w\xi(\pi).
\end{equation}
Here the action $^w\xi(\pi)$ of $W_M$ on a regular function $\xi$ is naturally defined  as $\xi(\, ^{w}\pi)$ (cf. Remark \ref{rem of poly}).
Then we have the following theorem which is a consequence of Proposition \ref{trace of f_phi and phi}.
\begin{theorem} \label{pro for phi and f}
Let $M$ be an $F$-Levi subgroup of $G$ corresponding to $\theta \subset \Delta,$ and let $\xi$ be a $W(\theta, \theta)$-invariant regular function on $\Omega(\sigma)$ (see Remarks \ref{rem of reg ft on omega} and \ref{W-invariant}). For every $\gamma \in M(F)^{\reg},$ we have
\[
|D_{G/M}(\gamma)| ^{1/2} \mathit{O}^{G(F)}_{\gamma}(f_{\sigma, \xi}) 
= \mathit{O}^{M(F)}_{\gamma}(\overline{f}_{\sigma, \xi}^{(P)})
= \mathit{O}^{M(F)}_{\gamma}(|W(\theta, \theta)| \cdot \phi_{\sigma, \xi}),
\]
where $D_{G/M}(\gamma) = \det(1-ad(\gamma)) |_{\Lie(G)/\Lie(M)}.$
\end{theorem}
\begin{proof} 
The first equality holds in general due to  \eqref{orbital int for par descent} in Section \ref{properties}. We shall check the second equality. 

When $M$ is not self-associate, $|W(\theta, \theta)| = 1.$ 
In this case, we do not need the condition of the $W(\theta, \theta)$-invariance. Thus the second equality is a consequence of \cite[Theorem 0]{kaz} and the fact that for any $\pi \in \Irr(M(F)),$
\begin{align*}
trace(\pi)(\bar{f}_{\sigma, \xi}^{(P)})
& = trace(\ii_{G,M} \pi )(f_{\sigma, \xi}) \\
& \overset{\eqref{equality for pro below}}{=} trace(\pi)( \phi_{\sigma, \xi}).
\end{align*}  

When $M$ is self-associate, we note that $|W(\theta, \theta)| \geq 2.$
Since $\xi$ is a $W(\theta, \theta)$-invariant by hypothesis, \eqref{equality for pro below} implies that
\[
trace(\pi)(\bar{f}_{\sigma, \xi}^{(P)}) = |W(\theta, \theta)| \cdot trace(\pi)(\phi_{\sigma, \xi})
\] 
for any $\pi \in \Irr(M(F)).$
Therefore, we have from \cite[Theorem 0]{kaz}
\[
\mathit{O}^{M(F)}_{\gamma}(\overline{f}_{\sigma, \xi}^{(P)})
= \mathit{O}^{M(F)}_{\gamma}(|W(\theta, \theta)| \cdot \phi_{\sigma, \xi}).
\] 
This completes the proof.
\end{proof}
\begin{remark} \label{remark for phi and f}
In the case that $|W(\theta, \theta)| \geq 2,$ given any regular function $\xi_o$ on $\Omega(\sigma)$ and any $w \in W(\theta, \theta),$ we set 
\[
^w \xi_o (z) = ^w \xi_o(\sigma \otimes \psi) := \xi_o(\, ^{w} \sigma \otimes \, ^{w} \psi).
\]
Then, $^w \xi_o (z)$ is clearly a regular function on $\Omega(\sigma).$ 
It turns out that
\[
\xi(z) := \dfrac{1}{|W(\theta, \theta)|}\dsum_{w \in W(\theta, \theta)} \, ^{w}\xi_o(z)
\]
must be a $W(\theta, \theta)$-invariant regular function on $\Omega(\sigma).$
It thus follows that there is always a $W(\theta, \theta)$-invariant regular function on $\Psi(M)$ (cf. Remark \ref{W-invariant}).
\end{remark}

\subsection{Application to the Plancherel formula} \label{f(1) and PM}
Given a discrete series representation $\sigma \in \Pi_{\disc}(M(F))$ and a regular function $\xi$ on $\Omega(\sigma),$
by means of the Plancherel formula, we see a relationship between the value $\phi_{\sigma, \xi}(1)$ and the formal degree $d_M(\sigma).$ 
We also express the value $f_{\sigma, \xi}(1)$ in terms of $\xi,$ the formal degree of $\sigma,$ and the Plancherel measure on $\Omega(\sigma).$

\begin{proposition}[Relationship between $\phi_{\sigma, \xi}(1)$ and $d_M(\sigma)$] \label{formula of phi(1)}
We have
\begin{equation} \label{phi in PF}
\phi_{\sigma, \xi}(1) = d_M(\sigma)  \int_{\psi \in \Psi_u(M) / {\Stab}_{\Psi}(\sigma)}  \xi(\sigma \otimes \psi) d \psi.
\end{equation}
\end{proposition}
\begin{proof}
We apply the generalized pseudo-coefficient $\phi_{\sigma, \xi}$ to the Plancherel formula (Theorem \ref{pf}). 
Since $G=M,$ we have $a(G|M)=1$ and $\mu_M (\pi)=1$ for any $\pi \in \Pi_{\disc}(M(F)).$ 
By the construction of $\phi_{\sigma, \xi}$ in Theorem \ref{phi}, we have
\[
trace(\pi)(\phi_{\sigma, \xi}) = 0
\]
unless $\pi \in \Omega(\sigma),$ in which case $trace(\pi)(\phi_{\sigma, \xi}) = \xi(\pi).$
Note that $\pi \in \Omega(\sigma)$ is of the form $\sigma \otimes \psi$ with $\psi \in \Psi_u(M),$ and $d_M(\sigma \otimes \psi)=d_M(\sigma)$ for any unitary character $\psi$ of $M(F).$
From the Plancherel formula (Theorem \ref{pf}), we then have
\[
\phi_{\sigma, \xi}(1)  =  d_M (\sigma) \int_{\pi \in \Pi_{\disc}(M(F)) \cap \Omega(\sigma)}  \xi(\pi) d \pi.
\]
Since $\Pi_{\disc}(M(F)) \cap \Omega(\sigma) \s  \Psi_u(M) / {\Stab}_{\Psi}(\sigma),$ we have \eqref{phi in PF}.
\end{proof}
\begin{remark}
If $\xi=1$ so that $\phi^\omega_{\sigma, \xi}$ becomes a pseudo-coefficient (Proposition \ref{pseudo}), the right-hand side of \eqref{phi in PF} becomes $d_M(\sigma)  \int_{\psi \in \Psi_u(M) / {\Stab}_{\Psi}(\sigma)} d \psi.$
By choosing a suitable measure $d \psi$ on $\Psi_u(M)$ so that $\int_{\psi \in \Psi_u(M) / {\Stab}_{\Psi}(\sigma)} d \psi=1$ (see Section \ref{def of PM}) ,
we deduce 
\begin{equation} \label{phi and formal degree}
\phi_{\sigma, \xi}(1)=d_M(\sigma).
\end{equation}
If the center $Z_M$ is anisotropic over $F,$ so that $\Psi(M)$ is trivial and the regular function $\xi$ on $\Omega(\sigma)$ equals just 1,
then \eqref{phi and formal degree} immediately follows. 
This is an analogue of the well-known property of a pseudo-coefficient $\phi^\omega_{\sigma, \xi}(1)=d_M(\sigma)$ 
(for example, see \cite[A.4]{dkv}).
\end{remark}
Recall the notation $\mu_{M}(\sigma \otimes \psi)$ for the Plancherel measure from Remark \ref{pm at L_F}. We then have the following proposition.
\begin{proposition}[Relationship between $f_{\sigma, \xi}(1)$ and $\mu_{M} (\sigma \otimes \psi)$] \label{formula of f(1)}
Let $M$ be an $F$-Levi subgroup of $G$ corresponding to $\theta \subset \Delta,$  and let $\xi$ be a $W(\theta, \theta)$-invariant regular function on $\Omega(\sigma).$ Then we have
\[
f_{\sigma, \xi}(1) = a(G|M) d_M(\sigma) |W(\theta, \theta)| \int_{\psi \in \Psi_u(M) / {\Stab}_{\Psi}(\sigma)} 
\mu_{M} (\sigma \otimes \psi) \xi(\sigma \otimes \psi) d \psi.
\]
\end{proposition}
\begin{proof}
Similar to the proof of Proposition \ref{formula of phi(1)}, we apply $f_{\sigma, \xi}$ to the Plancherel formula (Theorem \ref{pf}). 
Proposition \ref{trace of f_phi and phi} then yields
\begin{equation} \label{some equal in some pf}
f_{\sigma, \xi}(1) = a(G|M) \int_{\pi \in \Pi_{\disc}(M(F))} d_M(\pi) 
\mu_{M} (\pi) \dsum_{w \in W(\theta, \theta)} trace(\, ^w \pi)(\phi_{\sigma, \xi}) d \pi.
\end{equation}
Due to Theorem \ref{phi}, the isomorphism $\Pi_{\disc}(M(F)) \cap \Omega(\sigma) \s  \Psi_u(M) / {\Stab}_{\Psi}(\sigma),$ and the assumption that $\xi$ is $W(\theta, \theta)$-invariant, the right-hand side of \eqref{some equal in some pf} equals
\[
a(G|M) |W(\theta, \theta)| \int_{\psi \in \Psi_u(M) / {\Stab}_{\Psi}(\sigma)} 
d_M(\sigma \otimes \psi)
\mu_{M} (\sigma \otimes \psi) \xi(\sigma \otimes \psi) d \psi.
\]
Again, using the fact that $d_M(\sigma \otimes \psi)=d_M(\sigma)$ for any unitary character $\psi$ of $M(F),$ the proof is complete.
\end{proof}

\subsection*{Acknowledgements}
This work contains a part of the author's Ph.D. dissertation at Purdue University.
The author is very grateful and indebted to Professor Freydoon Shahidi for his incredible support.  
The author thanks Professors Laurent Clozel, David Goldberg, Gordan Savin, Sug Woo Shin, and Roger Zierau for their helpful comments on this work.
The author also thanks the anonymous referee for a careful reading and many helpful suggestions for improvement in the article. 
This work was partially supported by an AMS-Simons Travel Grant.


\begin{thebibliography}{DKV84}

\bibitem[AC89]{ac89}
James Arthur and Laurent Clozel, \emph{Simple algebras, base change, and the
  advanced theory of the trace formula}, Annals of Mathematics Studies, vol.
  120, Princeton University Press, Princeton, NJ, 1989. \MR{1007299
  (90m:22041)}

\bibitem[Art93]{art93}
James Arthur, \emph{On elliptic tempered characters}, Acta Math. \textbf{171}
  (1993), no.~1, 73--138. \MR{1237898 (94i:22038)}

\bibitem[BDK86]{bdk}
J.~Bernstein, P.~Deligne, and D.~Kazhdan, \emph{Trace {P}aley-{W}iener theorem
  for reductive {$p$}-adic groups}, J. Analyse Math. \textbf{47} (1986),
  180--192. \MR{874050 (88g:22016)}

\bibitem[BZ77]{bz77}
I.~N. Bernstein and A.~V. Zelevinsky, \emph{Induced representations of
  reductive {$p$}-adic groups. {I}}, Ann. Sci. \'Ecole Norm. Sup. (4)
  \textbf{10} (1977), no.~4, 441--472. \MR{0579172 (58 \#28310)}

\bibitem[Clo86]{clo86}
Laurent Clozel, \emph{On limit multiplicities of discrete series
  representations in spaces of automorphic forms}, Invent. Math. \textbf{83}
  (1986), no.~2, 265--284. \MR{818353 (87g:22012)}

\bibitem[Clo91]{clo89}
\bysame, \emph{Invariant harmonic analysis on the {S}chwartz space of a
  reductive {$p$}-adic group}, Harmonic analysis on reductive groups
  ({B}runswick, {ME}, 1989), Progr. Math., vol. 101, Birkh\"auser Boston,
  Boston, MA, 1991, pp.~101--121. \MR{1168480 (93h:22020)}

\bibitem[DKV84]{dkv}
P.~Deligne, D.~Kazhdan, and M.-F. Vign{\'e}ras, \emph{Repr\'esentations des
  alg\`ebres centrales simples {$p$}-adiques}, Travaux en Cours, Hermann,
  Paris, 1984, pp.~33--117. \MR{771672 (86h:11044)}

\bibitem[FLM14]{flm}
Tobias Finis, Erez Lapid, and Werner M\"uller, \emph{Limit multiplicities for
  principal congruence subgroups of {$GL(n)$} and {$SL(n)$}}, preprint;
  arXiv:1208.2257 [math.RT] (2014).

\bibitem[Har11]{harris11}
Michael Harris, \emph{An introduction to the stable trace formula}, On the
  stabilization of the trace formula, Stab. Trace Formula Shimura Var. Arith.
  Appl., vol.~1, Int. Press, Somerville, MA, 2011, pp.~3--47. \MR{2856366}

\bibitem[HC70]{hc162}
Harish-Chandra, \emph{Harmonic analysis on reductive {$p$}-adic groups},
  Lecture Notes in Mathematics, Vol. 162, Springer-Verlag, Berlin, 1970, Notes
  by G. van Dijk. \MR{0414797 (54 \#2889)}

\bibitem[HC73]{hc73}
\bysame, \emph{Harmonic analysis on reductive {$p$}-adic groups}, Amer. Math.
  Soc., Providence, R.I., 1973, pp.~167--192. \MR{0340486 (49 \#5238)}

\bibitem[HC81]{hc81}
\bysame, \emph{A submersion principle and its applications}, Proc. Indian Acad.
  Sci. Math. Sci. \textbf{90} (1981), no.~2, 95--102. \MR{653948 (83h:22031)}

\bibitem[HC84]{hc84}
\bysame, \emph{Collected papers. {V}ol. {IV}}, Springer-Verlag, New York, 1984,
  1970--1983, Edited by V. S. Varadarajan, pages 353--367. \MR{726026
  (85e:01061d)}

\bibitem[HC99]{hculs}
\bysame, \emph{Admissible invariant distributions on reductive {$p$}-adic
  groups}, University Lecture Series, vol.~16, American Mathematical Society,
  Providence, RI, 1999, Preface and notes by Stephen DeBacker and Paul J.
  Sally, Jr. \MR{1702257 (2001b:22015)}

\bibitem[HR10]{hai10}
Thomas~J. Haines and Sean Rostami, \emph{The {S}atake isomorphism for special
  maximal parahoric {H}ecke algebras}, Represent. Theory \textbf{14} (2010),
  264--284. \MR{2602034}

\bibitem[HS12]{hs11}
Kaoru Hiraga and Hiroshi Saito, \emph{On {$L$}-packets for inner forms of
  {$SL_n$}}, Mem. Amer. Math. Soc. \textbf{215} (2012), no.~1013, vi+97.
  \MR{2918491}

\bibitem[Kaz86]{kaz}
David Kazhdan, \emph{Cuspidal geometry of {$p$}-adic groups}, J. Analyse Math.
  \textbf{47} (1986), 1--36. \MR{874042 (88g:22017)}

\bibitem[Kot86]{kot86}
Robert~E. Kottwitz, \emph{Stable trace formula: elliptic singular terms}, Math.
  Ann. \textbf{275} (1986), no.~3, 365--399. \MR{858284 (88d:22027)}

\bibitem[Kot88]{kot88}
\bysame, \emph{Tamagawa numbers}, Ann. of Math. (2) \textbf{127} (1988), no.~3,
  629--646. \MR{942522 (90e:11075)}

\bibitem[Kot05]{kot05}
\bysame, \emph{Harmonic analysis on reductive {$p$}-adic groups and {L}ie
  algebras}, Harmonic analysis, the trace formula, and {S}himura varieties,
  Clay Math. Proc., vol.~4, Amer. Math. Soc., Providence, RI, 2005,
  pp.~393--522. \MR{2192014 (2006m:22016)}

\bibitem[MT11]{mt11}
Allen Moy and Marko Tadi{\'c}, \emph{A construction of elements in the
  {B}ernstein center for quasi-split groups}, Amer. J. Math. \textbf{133}
  (2011), no.~2, 467--518. \MR{2797354}

\bibitem[Roc09]{roc09}
Alan Roche, \emph{The {B}ernstein decomposition and the {B}ernstein centre},
  Ottawa lectures on admissible representations of reductive {$p$}-adic groups,
  Fields Inst. Monogr., vol.~26, Amer. Math. Soc., Providence, RI, 2009,
  pp.~3--52. \MR{2508719 (2010c:22022)}

\bibitem[Sha90]{sh90}
Freydoon Shahidi, \emph{A proof of {L}anglands' conjecture on {P}lancherel
  measures; complementary series for {$p$}-adic groups}, Ann. of Math. (2)
  \textbf{132} (1990), no.~2, 273--330. \MR{1070599 (91m:11095)}

\bibitem[Shi10]{shin10}
Sug~Woo Shin, \emph{A stable trace formula for {I}gusa varieties}, J. Inst.
  Math. Jussieu \textbf{9} (2010), no.~4, 847--895. \MR{2684263}

\bibitem[Sil79]{sil79}
Allan~J. Silberger, \emph{Introduction to harmonic analysis on reductive
  {$p$}-adic groups}, Mathematical Notes, vol.~23, Princeton University Press,
  Princeton, N.J., 1979, Based on lectures by Harish-Chandra at the Institute
  for Advanced Study, 1971--1973. \MR{544991 (81m:22025)}

\bibitem[SS97]{ss}
Peter Schneider and Ulrich Stuhler, \emph{Representation theory and sheaves on
  the {B}ruhat-{T}its building}, Inst. Hautes \'Etudes Sci. Publ. Math. (1997),
  no.~85, 97--191. \MR{1471867 (98m:22023)}

\bibitem[vD72]{dij72}
G.~van Dijk, \emph{Computation of certain induced characters of {$p$}-adic
  groups}, Math. Ann. \textbf{199} (1972), 229--240. \MR{0338277 (49 \#3043)}

\bibitem[Wal03]{wal03}
J.-L. Waldspurger, \emph{La formule de {P}lancherel pour les groupes
  {$p$}-adiques (d'apr\`es {H}arish-{C}handra)}, J. Inst. Math. Jussieu
  \textbf{2} (2003), no.~2, 235--333. \MR{1989693 (2004d:22009)}

\end{thebibliography}
\end{document}